\documentclass[11pt]{article}
\usepackage{amssymb,amsmath,amsthm,bm,mathrsfs,paralist}
\usepackage{url,color,units,graphicx,mathtools,bbold}
\usepackage[T1]{fontenc}
\usepackage[margin=1in]{geometry}
\RequirePackage[colorlinks,citecolor=blue,urlcolor=blue]{hyperref}
\usepackage{pdfsync}
\usepackage{cite}

\newcommand{\I}{\mathcal{I}}
\renewcommand{\L}{\mathscr{L}}
\newcommand{\Lloc}{\mathscr{L}_{\mathit{loc}}}
\newcommand{\R}{\mathbb{R}}
\newcommand{\lip}{\text{\rm Lip}}

\newcommand{\dW}{\dot{W}}

\renewcommand{\d}{\mathrm{d}}
\renewcommand{\P}{\mathrm{P}}
\newcommand{\e}{\mathrm{e}}
\newcommand{\<}{\langle}
\renewcommand{\>}{\rangle}
\newcommand{\E}{\mathrm{E}}

\DeclareMathOperator{\Cov}{\mathrm{Cov}}

\newtheorem{proposition}{Proposition}
\newtheorem{theorem}[proposition]{Theorem}
\newtheorem{lemma}[proposition]{Lemma}
\newtheorem{corollary}[proposition]{Corollary}
\newtheorem{definition}[proposition]{Definition}
\newtheorem{assumption}[proposition]{Assumption}
\newtheorem{condition}[proposition]{Condition}

\theoremstyle{definition}
\newtheorem{remark}[proposition]{Remark}

\numberwithin{equation}{section}
\numberwithin{proposition}{section}

\title{Instantaneous everywhere-blowup of parabolic SPDEs}
\author{Mohammud Foondun\\University of Strathclyde
\and Davar Khoshnevisan\thanks{Research supported in part by the United States' National Science Foundation grants DMS-1855439 and DMS-2245242}\\The University of Utah
\and Eulalia Nualart\thanks{Acknowledges support from the Spanish MINECO grant PGC2018-101643-B-I00 and
Ayudas Fundacion BBVA a Proyectos de Investigaci\'on Cient\'ifica 2021}\\Universitat Pompeu Fabra}

\begin{document}
\maketitle

\begin{abstract}
\noindent We consider the following stochastic heat equation
\begin{equation*}
	\partial_t u(t\,,x) = \tfrac12 \partial^2_x u(t\,,x) + b(u(t\,,x)) + \sigma(u(t\,,x)) \dot{W}(t\,,x),
\end{equation*}
defined for $(t\,,x)\in(0\,,\infty)\times\R$, where $\dot{W}$ denotes space-time white noise. 
The function $\sigma$ is assumed to be positive, bounded, globally Lipschitz, and bounded 
uniformly away from the origin, and the function $b$ is assumed to be positive, locally Lipschitz and nondecreasing. 
We prove that the Osgood condition
\[
	\int_1^\infty\frac{\d y}{b(y)}<\infty
\]
implies that the solution almost surely blows up everywhere and instantaneously, 
In other words, the Osgood condition ensures that
$\P\{ u(t\,,x)=\infty\quad\text{for all $t>0$ and $x\in\R$}\}=1.$
The main ingredients of the proof involve a hitting-time bound for
a class of differential inequalities (Remark \ref{rem:diff_ineq}), and the study of the spatial growth of stochastic 
convolutions using techniques from the Malliavin calculus and the Poincar\'e inequalities 
that were developed in Chen et al \cite{CKNP, CKNP1}.\vskip 24pt

	
\end{abstract}

\section{Introduction}

We consider the following stochastic heat equation
\begin{equation}\label{SHE}\left[\begin{split}
	&\partial_t u(t\,,x) = \tfrac12 \partial^2_x u(t\,,x) + b(u(t\,,x)) + \sigma(u(t\,,x)) \dot{W}(t\,,x)
		&\text{for $(t\,,x)\in(0\,,\infty)\times\R$},\\
	&\text{subject to }u(0\,,x) = u_0(x)&\text{for all $x\in\R$}.
\end{split}\right.\end{equation}
The initial condition $u_0$ is assumed to be a non-random bounded function,
and the noise term is space-time white noise; that is, 
$\dot{W}$ is a centered, generalized Gaussian random field with
\begin{equation*}
	\Cov [ \dW(t\,,x) \,, \dW(s\,,y) ] = \delta_0(t-s) \delta_0(x-y)
	\quad\text{for all $t,s\ge0$ and $x,y\in\R$}.
\end{equation*}
Throughout, we assume that $\sigma$ and $b$ satisfy the following hypotheses:
\begin{assumption}\label{cond-dif}
	$\sigma:\R \rightarrow (0, \infty)$ is  Lipschitz continuous, and satisfies
	$0<\inf_\R \sigma \leq \sup_\R \sigma <\infty.$
\end{assumption}
\begin{assumption}\label{cond-drift}
	$b:\R \to(0,\,\infty)$ is locally Lipschitz continuous, as well as nondecreasing. 
\end{assumption}
We recall that a random field solution to \eqref{SHE} is a predictable random field  $u=\{u(t\,,x)\}_{t \geq 0, x \in \R}$ that satisfies the following integral equation:
\begin{equation}\label{mild}
	u(t\,,x) = (p_t*u_0)(x) + \int_{(0,t)\times\R} p_{t-s}(y-x) b(u(s\,,y))\,\d s\,\d y
	+ \I(t\,,x),
\end{equation}
where 
\begin{equation*}
	\I(t\,,x) = \int_{(0,t)\times\R} p_{t-s}(y-x) \sigma(u(s\,,y)) \, W(\d s\,\d y),
\end{equation*}
the symbol $*$ denotes convolution, and
\begin{equation*}
	p_r(z) = \frac{\exp\{-z^2/(2r)\}}{\sqrt{2\pi r}}\qquad\text{for all $r>0$ and $z\in\R$}.
\end{equation*}

When $b$ and $\sigma$ are  Lipschitz continuous, general theory ensures that the
SPDE \eqref{mild} is well posed; see Dalang \cite{Dalang} and Walsh \cite{Walsh}. 
However, general theory fails to be applicable when $b$ and/or $\sigma$ are  assumed
to be only locally Lipschitz continuous.  Here, we can exploit the fact that $b$ is nondecreasing 
in order to ensure the existence of a ``minimal solution'' $u$ under Assumptions
\ref{cond-dif} and \ref{cond-drift}; see the beginning of the proof of Theorem 1.5 in Section 5 for more details. 
With that under way, we turn to the main objective of this paper and 
prove that, under Assumptions \ref{cond-dif} and \ref{cond-drift}, the classical Osgood condition 
\eqref{Osgood} of ODEs ensures that the minimal solution, and hence every solution,
to \eqref{SHE} blows up everywhere and instantaneously.  

There is a large and distinguished literature in PDEs that focuses on these types of questions; see for example 
Cabr\'e and Martel \cite{CM99}, Peral and V\'azquez \cite{PV95}, and V\'azquez \cite{V99}. 
To the best of our knowledge, the present paper contains the first instantaneous blowup 
result for SPDEs of the type given by \eqref{SHE}.  For PDEs, various different definitions for instantaneous blowup are used but all these notions basically mean that the solution blows up for every $t>0$. We provide 
a different definition that is particularly well suited for our purposes.
\begin{definition}
	Let $u=\{u(t\,,x)\}_{t \geq 0, x \in \R}$ denote a space-time random field with values in
	$[-\infty\,,\infty]$. We say that $u$ \emph{blows up everywhere and instantaneously}
	when
	\[
		\P\left\{ u(t\,,x)=\infty\text{ for every $t>0$ and $x\in\R$}\right\}=1.
	\]
\end{definition}
Our notion of instantaneous, everywhere blowup is sometimes referred to as 
\emph{instantaneous and complete blowup}. 

We are not aware of any prior known results on instantaneous nor everywhere blowup 
in the SPDE literature.  However, broader questions of blowup for SPDEs have received
recent attention. Recent examples include Ref.s \cite{BonderG, Dal-Khos-Zhang, FN2, Foondun-Nualart2},
where criteria for the  blowup in finite time with positive probability or almost surely are studied. 
And De Bouard and Debussche \cite{DD05} investigate blowup for the stochastic nonlinear Schr\"odinger equation,
valid in arbitrarily small time, and with  positive probability; see also the references in \cite{DD05}.

In order to state our result precisely, we need the well-known Osgood condition from the classical theory of
ODEs.
\begin{condition}\label{Cond-Osgood}
A function $b:\R\mapsto (0\,,\infty)$ is said to satisfy the Osgood condition if 
	\begin{equation}\label{Osgood}
		\int_1^\infty\frac{\d y}{b(y)}<\infty,
	\end{equation}
	where $1/0=\infty$.
\end{condition}

It was proved in Foondun and Nualart \cite{FN2} that, when $\sigma$
is a positive constant, the Osgood condition implies that the solution to  \eqref{SHE} 
blows up almost surely.  Earlier, this fact was previously proved by Bonder and Groisman \cite{BonderG} 
for SPDEs on a finite interval. In the converse direction, and for the same equations on finite
intervals, Foondun and Nualart \cite{FN2}
have shown that if $\sigma$ is locally Lipschitz continuous and bounded, then the Osgood condition is  
necessary for the solution to blow up somewhere with positive probability.

Recall Assumptions \ref{cond-dif} and \ref{cond-drift}.
The aim of the present paper is to show that the Osgood condition in fact
implies that, almost surely, the solution to equation \eqref{SHE} blows up everywhere and instantaneously.

\begin{theorem}\label{th:blowup}
	If $b$ satisfies the Osgood Condition \ref{Cond-Osgood},
	then the minimal solution to \eqref{SHE} blows up everywhere and instantaneously almost surely.
\end{theorem}

A few years ago, Professor Alison Etheridge asked one of us a number of questions
about the time to blow up and the nature of the blowup for stochastic reaction-diffusion equations of
the general type studied here. This paper provides the answer to Professor Etheridge's
questions in the case that $\sigma$ satisfies Assumption \ref{cond-dif}.  We do not have 
sharp blowup results when Assumption \ref{cond-dif} fails. Perhaps a noteworthy example
is $\sigma(u)=u$, which lies well outside the present theory.

We now describe the main strategy behind the the proof of Theorem \ref{th:blowup}.  
We may recast \eqref{mild} as 
\begin{equation*}
	u = \text{Term A}+\text{Term B}+\text{Term C},
\end{equation*} 
notation being clear.
Term A  is  deterministic, involves the initial condition, and plays no role in the blowup phenomenon
because the initial condition is a nice function. In the PDE literature, there are many results about blowup 
that hold because the initial condition is assumed to be singular. Here, the initial data is 
a very nice function with no singularities. In our setting, blowup occurs for very different
reasons, and is caused by the interplay between the stochastic Term B, which is the highly non-linear term, and 
the other stochastic Term C, which is regarded as a Walsh stochastic integral.  Next, we will say a few words about
this interplay.

As part of our analysis, we  prove that, when $b$ is in fact a Lipschitz continuous function
that satisfies the Osgood condition \eqref{Osgood}, the process $x\mapsto u(t\,,x)$ is almost surely unbounded 
for every $t>0$.  The proof of this fact makes use of ideas from the Malliavin calculus 
and Poincar\'e inequalities developed in a recent paper by Chen et al \cite{CKNP}. 
The limiting procedure used to define the solution then allows us to use the growth property of $b$ to show blowup of the solution and thus complete the proof of the main result.  

We end this introduction with a plan of the paper. In \S2 we study ergodicity and 
growth properties for a family of stochastic convolutions. In \S3 we use some of these 
results to show that, when $b$ is  Lipschitz and the initial condition is a constant, 
the solution to \eqref{SHE} is spatially stationary and ergodic. 
In \S4 we develop a hitting-time estimate for a family of differential inequalities and
subsequently use that estimate
in order to obtain a  lower bound for $u$. The remaining details of the proof of Theorem \ref{th:blowup} are
gathered  in \S5, using the earlier results of the paper.

Throughout this paper, we write 
\[
	\|X\|_p = \left\{ \E(|X|^p)\right\}^{1/p}\qquad\text{for all $p\ge1$ and $X\in L^p(\Omega)$}.
\]
For every function $f:\R\to\R$, $\lip(f)$ denotes the optimal Lipschitz constant of $f$; that is,
\[
	\lip(f) = \sup_{-\infty<a<b<\infty}\frac{|f(b)-f(a)|}{b-a}.
\]
In particular, $f$ is Lipschitz continuous iff $\lip(f)<\infty$.

\section{Spatial growth of stochastic convolutions}

\subsection{Spatial ergodicity via the Malliavin calculus}

We introduce following Nualart \cite{N} some elements of 
the Malliavin calculus that we will need. Let $\mathcal{H}=L^2(\R_+ \times  \R)$. For
every Malliavin-differentiable random variable $F$,
we let $DF$ denote the Malliavin derivative of $F$, and observe that $DF=\{ D_{r,z}F\}_{r>0,z\in\R}$ is
a random field indexed by $(r\,, z)\in \R_+\times \R$. 

For every $p \geq 2$, let $\mathbb{D}^{1,p}$ denote the usual Gaussian Sobolev space endowed with the semi-norm 
\[
	\|F\|_{1,p}^p:=\E(|F|^p)+\E(\|DF\|^p_\mathcal{H}).
\]
We will need the following version of the Poincar\'e inequality due to Chen et al \cite[(2.1)]{CKNP}:
\begin{equation} \label{poincare}
	\vert \text{Cov} (F\,,G) \vert \leq \int_0^{\infty} \d r
	\int_{-\infty}^{\infty} \d z\ \Vert D_{r,z} F \Vert_2 \Vert D_{r,z} G \Vert_2
	\qquad\text{for every $F,G$ in $\mathbb{D}^{1,2}$}.
\end{equation}

Next, let us recall some notions from the ergodic theory of multiparameter processes 
(see for example Chen et al \cite{CKNP1}):
We say that a predictable random field $Z=\{Z(t\,,x)\}_{(t,x)\in(0,\infty)\times\R}$ is 
\emph{spatially mixing }
when the random field $x \rightarrow Z(t\,,x)$ is weakly mixing in the usual sense 
for every $t>0$. This property can be stated as follows: For all $k\in\mathbb{N}$, $t>0$, 
$\xi^1,...,\xi^k\in\R$, and Lipschitz-continuous functions $g_1,...,g_k :\R \rightarrow \R$ 
that satisfy $g_j(0)=0$ and Lip$(g_j)=1$ for every $j =1,...,k$, 
\begin{equation} \label{cov}
\lim_{\vert x \vert \rightarrow \infty} \text{Cov} [\mathcal{G}(x)\,, \mathcal{G}(0)]=0,
\end{equation}
where
\begin{equation} \label{G}
\mathcal{G}(x)=\prod_{j=1}^k g_j(Z(t,x+\xi^j)), \quad x \in \R.
\end{equation}
Whenever the process $x \rightarrow Z(t\,,x)$ is stationary and weakly mixing for all $t>0$,  it is ergodic.

Finally, we will require two elementary identities for products of the heat kernel. Namely, that
\begin{equation} \label{identity0}
	p_{t-s}(x-y) p_{s}(y-z)=p_t(x-z)p_{s(t-s)/t}\left(y-z-\frac{s}{t}(x-z) \right),
\end{equation}
and
\begin{equation} \label{identity}
	\int_{-\infty}^{\infty} \left[ p_{t-s}(x-y)\right]^2 \left[ p_{s-r}(y-z)\right]^2\, \d y=
	\sqrt{\frac{t-r}{4\pi (t-s) (s-r)}}\left[ p_{t-r}(x-z)\right]^2.
\end{equation}
See Chen et al  \cite[below (6.10)]{CKNP1} for \eqref{identity0} and 
Chen et al \cite[below (2.7)]{CKNP} for \eqref{identity}.

\subsection{Ergodicity of stochastic convolutions}

Let $Z=\{Z(t\,,x)\}_{(t,x)\in(0,\infty)\times\R}$ be a predictable random field that satisfies
\begin{equation}\label{cZc}
	c_1 \le \inf_{(t,x)\in(0,\infty)\times\R}  Z(t\,,x) 
	\le \sup_{(t,x)\in(0,\infty)\times\R}  Z(t\,,x)  \le c_2,
\end{equation}
for two positive and finite constants $c_1$ and $c_2$ that are fixed throughout.
Set $I_Z(0\,,x)=0$, and consider the associated stochastic convolution
\begin{equation}\label{eq:I_Z}
	I_Z(t\,,x) = \int_{(0,t)\times\R} p_{t-s}(y-x) Z(s\,,y)\, W(\d s\,\d y)\qquad
	\text{for every $t>0$ and $x\in\R$}.
\end{equation}
The main aim of this section is to study the growth properties of the
random field $x \rightarrow I_Z(t\, ,x)$. Next we develop natural conditions 
under which the random field $x \rightarrow I_Z(t\,,x)$ is stationary and ergodic
at all times $t>0$. 

\begin{proposition}
	Assume that $x \rightarrow Z(t\,, x)$ is stationary for all $t>0$. Assume also that
	$Z(t\,,x) \in \mathbb{D}^{1,p}$
	for all $p \geq 2$, $t>0$ and $x \in \R$, and that its Malliavin derivative $DZ(t\,,x)$
	has the following property: 
	For every $T>0$ and $p \geq 2$ there exists a number $C_{T,p}>0$ such that 
	\begin{equation} \label{DZ}
	\Vert D_{r,z} Z(t\,,x) \Vert_p \leq C_{T,p}\, p_{t-r}(x-z) p_r(z),
	\end{equation}
	for every
	$t \in (0\,,T)$ and $x \in \R$ and for almost every $(r\,,z) \in (0\,,t) \times \R$.
	Then the process $x \rightarrow Z(t\,, x)$ is ergodic for every $t>0$, 
	and $x \rightarrow I_Z(t\,, x)$ is stationary and ergodic for every $t>0.$
\end{proposition}

\begin{proof}
	Thanks to the Poincar\'e inequality \eqref{poincare},
	the proof of ergodicity follows the same pattern as  \cite[Proof of Theorem 1.3]{CKNP1}.
	Therefore, we describe the argument quickly mainly where adjustments are needed.
	
	We start with the process $Z$ and use a similar argument as
	Chen et al \cite[Proof of Corollary 9.1]{CKNP1}; see also Chen et al \cite[Theorem 1.1]{CKNP}.
	Define $\mathcal{G}(x)$ as was done in \eqref{G}.  It then follows from \eqref{DZ} 
	and \eqref{identity0} that there exists a constant $c_{T,k}>0$ such that
	\begin{equation*}\begin{split}
		\Vert D_{r,z} \mathcal{G}(x) \Vert_2 
		&\leq \sum_{j_0=1}^k \left(\prod_{j=1, j \neq j_0}^k \Vert g_j(Z(t\,, x+\xi^j)) \Vert_{2k} \right) \Vert D_{r,z} Z(t\,,x+\xi^{j_0}) \Vert_{2k} \\
		&\leq c_{T,k} \sum_{j=1}^k
		p_{t-r}(x+\xi^j-z) p_r(z) \\
		&\leq c_{T,k} \sum_{j=1}^k
		p_t(x+\xi^j) p_{r(t-r)/t} \left(z-\frac{r}{t}(x+\xi^j)\right),
	\end{split}\end{equation*}
	valid uniformly for all $0<r<t\leq T$ and $x,z \in \R$.\footnote{The notation $c_{t,k}$ may
	refer to a constant that changes from line to line but in any case depends only on $(t\,,k)$.}
	
	We can combine the Poincar\'e inequality \eqref{poincare}, the heat-kernel
	identity \eqref{identity0}, and the semigroup property of the heat kernel to find that
	\begin{equation*}
		\vert \text{Cov} [\mathcal{G}(x), \mathcal{G}(0)] \vert 
		\leq c_{T,k} \sum_{j,\ell=1}^k p_t(x+\xi^j) p_t(x+\xi^{\ell})\int_0^t 
		p_{2r(t-r)/t}\left (\frac{r}{t}(x+\xi^j-\xi^{\ell})\right)\, \d r.
	\end{equation*}
	Therefore, the dominated convergence  implies (\ref{cov}), whence follows the
	 ergodicity of $x \rightarrow Z(t\,,x)$ for every $t>0$.
	
	Next, we show that the process $x \rightarrow I_Z(t\,,x)$ is stationary for all $t>0$.
	The proof of this fact follows the proof of Lemma 7.1 in \cite{CKNP1} closely. 
	First, let us choose and fix some $y \in \R$
	and apply (7.2) in \cite{CKNP1} as follows:
	\begin{equation*} \begin{split}
	(I_Z\circ \theta_y)(t\,,x)= I_Z(t\,,x+y) &= \int_{(0,t)\times\R} p_{t-s}(x+y-z) Z(s\,,z-y+y)\, W(\d s\,\d z) \\
	&= \int_{(0,t)\times\R} p_{t-s}(x-z) Z(s\,,z+y)\, W_y(\d s\,\d z)\\
	&= \int_{(0,t)\times\R} p_{t-s}(x-z) (Z\circ \theta_y)(s\,,z)\, W_y(\d s\,\d z),
	\end{split}
	\end{equation*}
	where $\theta_y$ denotes the shift operator (see Chen et al \cite{CKNP1}),
	and $W_y$ is the associated
	shifted Gaussian noise \cite[(7.1)]{CKNP1}. The spatial stationarity of $I_Z$
	follows from the facts that $W$ and $W_y$ have the same law 
	and the random field $Z\circ \theta_y$ has the same finite-dimensional distributions as $Z$ because $Z$ is assumed to be spatially stationary.
	
	We now turn to the spatial ergodicity of the process $I_Z$.  
	By the properties of the divergence operator \cite[Proposition 1.3.8]{N}, 
	$I_Z(t\,,x) \in \mathbb{D}^{1,k}$ for all $k \geq 2$, $t>0$, and $x \in \R$.
	Moreover, the Malliavin derivative $DI_Z(t\,,x)$ a.s.\ satisfies 
	 \begin{equation*} 
		D_{r,z} I_Z(t\,,x) = p_{t-r}(x-z) Z(r\,,z) +
		\int_{(r,t)\times\R} p_{t-s}(y-x) D_{r,z} Z(s\,,y)  \, W(\d s\,\d y).
	\end{equation*}
	In principle, the above is valid for a.e.\ $(r\,,z)$ but in fact the right-hand side can be used to
	define the Malliavin derivative everywhere a.s. And that is what we do here.
	In particular,
	for any integer $k \geq 2$, the Burkholder-Davis-Gundy inequality and the estimate (\ref{DZ}) 
	together imply that
	\begin{equation*} \begin{split}
		\Vert D_{r,z} I_Z(t\,,x) \Vert_{2k}&\leq c p_{t-r}(x-z) +
			c_k\left(\int_r^t \d s \int_{\R} \d y \left[p_{t-s}(x-y)\right]^2
			\Vert D_{r,z} Z(s,y) \Vert^2_{2k}\right)^{1/2} \\
		&\leq c p_{t-r}(x-z) +
	 		c_{T,k}\left(\int_r^t \d s \int_{\R} \d y 
			\left[p_{t-s}(x-y)\right]^2 \left[p_{s-r}(y-z)\right]^2 \left[ p_r(z)\right]^2\right)^{1/2}.
		\end{split}
	\end{equation*}
	Thanks to \eqref{identity}, this yields
	\begin{equation}\label{DB}\begin{split}
		\Vert D_{r,z} I_Z(t\,,x) \Vert_{2k} &\leq c p_{t-r}(x-z) +
		 	c_{T,k} p_r(z) p_{t-r}(x-z) \left(\int_r^t  \sqrt{\frac{t-r}{4\pi(t-s)(s-r)}}\ \d s\right)^{1/2} \\
	 	& \leq c_{T,k} p_{t-r}(x-z) (1+ p_r(z) (t-r)^{1/4}).
	\end{split}\end{equation}
	Define 
	\[
		\mathcal{J}(x)=\prod_{j=1}^k g_j(I_Z(t\,,x+\xi^j))
		\qquad\text{for $x\in\R$},
	\]
	using the same $g^1,\ldots,g^k$ and $\xi^1,\ldots,\xi^k$ that were introduced earlier.
	In this way we can conclude from \eqref{DB} and elementary properties of the Malliavin derivative
	that
	\begin{equation*}\begin{split}
		\Vert D_{r,z} \mathcal{J}(x) \Vert_2 
		&\leq \sum_{j_0=1}^k \left(\prod_{j=1, j \neq j_0}^k \Vert g_j(I_Z(t\,, x+\xi^j)) \Vert_{2k} \right) \Vert D_{r,z} I_Z(t\,,x+\xi^{j_0}) \Vert_{2k} \\
		&\leq c_{T,k} \sum_{j=1}^k
			p_{t-r}(x+\xi^j-z) (1+ p_r(z) (t-r)^{1/4}) \\
		&=c_{T,k} \sum_{j=1}^k \left[
			p_t(x+\xi^j-z)+ p_{r(t-r)/t}\left(z-\frac{r}{t}(x+\xi^j)\right)p_t(x+\xi^j)(t-r)^{1/4}\right],
	\end{split}\end{equation*}
	valid uniformly for all $0<r<t\leq T$ and $x,z \in \R$.
	
	Now we apply \eqref{poincare} together with the semigroup property of the heat kernel to see that
	\begin{equation*} \begin{split}
		\vert \text{Cov} [\mathcal{J}(x), \mathcal{J}(0)] \vert& 
			\leq c_{T,k} \sum_{j,\ell=1}^k \bigg[t p_{2t}(x+\xi^j-\xi^{\ell})\\
			&\qquad +
			\int_0^t p_{t+\frac{r(t-r)}{t}}\left(x+\xi^j-\frac{r}{t}\xi^{\ell}\right)p_t(\xi^{\ell})(t-r)^{1/4}\,\d r\\
		&\qquad +\int_0^t p_{t+\frac{r(t-r)}{t}}\left(\frac{r}{t}(x+\xi^j)-\xi^{\ell}\right)p_t(x+\xi^{j})(t-r)^{1/4}\,\d r\\
		&\qquad + p_t(x+\xi^j) p_t(x+\xi^{\ell})\int_0^t 
			p_{2r(t-r)/t}\left (\frac{r}{t}(x+\xi^j-\xi^{\ell})\right) (t-r)^{1/4}\, \d r\bigg].
	\end{split}\end{equation*}
	The dominated convergence  implies that
	$\lim_{\vert x \vert \rightarrow \infty} \text{Cov} [\mathcal{J}(x)\,, \mathcal{J}(0)]=0$,
	and hence follows the ergodicity of $x \rightarrow I_Z(t\,,x)$ for every $t>0$.
	This concludes the proof.
\end{proof}

\subsection{Spatial growth of stochastic convolutions}

We are ready to state the main result of this section. 
\begin{theorem}\label{th:key}
	Choose and fix $c_2>c_1>0$. Then, there exists $\eta=\eta(c_1\,,c_2)>0$ such that
	\[
		\P\left\{ \limsup_{c\to\infty}\adjustlimits\inf_{t\in(a,a+(\eta a)^2)}\inf_{x\in(0,\eta a)}
		I_Z(t\,,c+x)=\infty\right\}=1,
	\]
	valid for every non-random number $a>0$ and every predictable random field $Z$ that
	satisfies the boundedness condition \eqref{cZc} and for which 
	$x\mapsto I_Z(t\,,x)$ is stationary and ergodic for all $t>0$.
\end{theorem}

\begin{remark}
	Note, in particular, that the constant $\eta$ does not
	depend on the choice of $Z$. This is the crucial part of the message of
	Theorem \ref{th:key}.
\end{remark}

The proof of Theorem \ref{th:key} requires a few prefatory steps that we present
as a series of lemmas. Once those lemmas are under way, we are able to prove Theorem
\ref{th:key} promptly.

\begin{lemma}\label{lem:tails}
	For every $c_2>c_1>0$ there exist $C_2,C_1>0$ such that 
	\[
		\frac{C_1}{1+\lambda} \exp\left(- \frac{\lambda^2}{2c_1^2}\right) \le
		\P\left\{ I_Z(t\,,x)  \ge (t/\pi)^{1/4}\lambda\right\}
		\le \frac{C_2}{1+\lambda} \exp\left(- \frac{\lambda^2}{2c_2^2}\right),
	\]
	uniformly for all $t,\lambda\ge0$ and $x\in\R$, and for every predictable random field
	$Z$ that satisfies \eqref{cZc}.
\end{lemma}

\begin{proof}
	Choose and fix $t>0$ and consider
	\begin{equation*}
		M_0=0
		\quad\text{and}\quad	
		M_r = \int_{(0,r)\times\R} p_{t-s}(y-x) Z(s\,,y)\,W(\d s\,\d y)
		\qquad\text{for $0<r\le t$}.
	\end{equation*}
	Because $Z$ is uniformly bounded, the above is a continuous, $L^2$-martingale with
	quadratic variation
	\[
		\<M\>_r = \int_0^r\d s\int_{-\infty}^\infty\d y\
		[p_{t-s}(y-x)]^2 |Z(s\,,y)|^2	
		\qquad\text{for $0\le r\le t$}.
	\]
	Because
	\[
		\int_0^r\d s\int_{-\infty}^\infty\d y\ [p_{t-s}(y-x)]^2 = \int_0^r \frac{\d s}{\sqrt{4\pi(t-s)}}
		=  \sqrt{\frac t\pi} - \sqrt{\frac{t-r}{\pi}}		\qquad\text{for $0\le r\le t$},
	\]
	the inequalities \eqref{cZc} yield 
	\begin{equation}\label{<M>}
		\frac{c_1^2}{\sqrt\pi}\left[\sqrt t - \sqrt{t-r}\right]\le
		\< M\>_r \le \frac{c_2^2}{\sqrt\pi}\left[\sqrt t - \sqrt{t-r}\right]
		\qquad\text{for $0\le r\le t$}.
	\end{equation}
	The Dubins, Dambis-Schwartz theorem, see \cite{RY}, ensures that $M_r = B(\<M\>_r)$
	for a standard, linear Brownian motion $B$. Since $I_Z(t\,,x)=M_t$ is the terminal
	point of our martingale $M$, and because \eqref{<M>} implies
	that $\<M\>_t\le c_2^2\sqrt{t/\pi}$, we learn from the reflection principle 
	and the scaling property that
	\[
		\P\left\{ I_Z(t\,,x)  \ge c_2(t/\pi)^{1/4}\lambda\right\}
		\le \P\left\{ \sup_{0\le r\le c_2^2\sqrt{t/\pi}} B(r) \ge c_2(t/\pi)^{1/4}\lambda\right\}
		=\sqrt{2/\pi}\int_{\lambda}^\infty \e^{-z^2/2}\,\d z. 
	\]
	A standard estimate yields the upper bound. For the lower bound we  observe in like manner to the preceding that
	\begin{align*}	
		&\P\left\{ I_Z(t\,,x)  \ge c_1(t/\pi)^{1/4}\lambda\right\}\\
		&\ge \P\left\{ B\left( c_1^2\sqrt{t/\pi} \right) \ge 2
			c_1(t/\pi)^{1/4}\lambda\right\}\P\left\{ \sup_{\nu\in[c_1^2,c_2^2]}
			\left| B\left( \nu\sqrt{t/\pi}\right) - B\left( c_1^2\sqrt{t/\pi}\right)
			\right| \le  c_1  (t/\pi)^{1/4}\right\}\\
		&= \frac{\varpi}{\sqrt{2\pi}} \int_{2\lambda}^\infty\e^{-z^2/2}\,\d z,
	\end{align*}
	where  $\varpi = \P\{ \sup_{\nu\in[1,(c_2/c_1)^2]} | B(\nu)-B(1)| \le 1\}\in(0\,,1).$ This proves that
	\[
		\P\left\{ I_Z(t\,,x)  \ge c_1(t/\pi)^{1/4}\lambda\right\} \gtrsim
		\lambda^{-1}\exp(-\lambda^2/2)\qquad\text{for all $\lambda\ge1$},
	\]
	where the implied constant depends only on $c_1$ and $c_2$. When $\lambda\in(0\,,1)$, it suffices to lower bound the integral by a constant.
\end{proof}

\begin{lemma}\label{lem:inc:x}
	Choose and fix  a non-random number $c_0>0$. Then,
	\[
		\adjustlimits
		\sup_{t\ge0}\sup_{-\infty<x\neq z<\infty}\E\left( \left|
		\frac{I_Z(t\,,x) - I_Z(t\,,z)}{|x-z|^{1/2}}\right|^k \right) \le (2 c_0^2k)^{k/2},
	\]
	for every $k\in[2\,,\infty)$ and for all predictable random fields
	$Z$ that satisfy $\sup_{p\in\R_+\times\R}|Z(p)|\le c_0$.
\end{lemma}

\begin{remark}
	We emphasize that Lemma \ref{lem:inc:x} assumes that $Z$ is bounded.
	This is a much weaker condition than \eqref{cZc}, as the latter implies
	also that, among other things, $\inf_{p\in\R_+\times\R}Z(p)$ is a.s.\
	bounded from below by a strictly positive, deterministic number. The next
	lemmas also in fact require only this weaker boundedness condition.
\end{remark}

\begin{proof}
	Choose and fix $t\ge0$ and $x\neq z\in\R$, and let $Z$
	be as described. By the Burkholder-Davis-Gundy inequality in the form \cite{minicourse},
	for every real number $k\ge 2$,
	\begin{align*}
		\| I_Z(t\,,x) - I_Z(t\,,z) \|_k^2 &\le 4k\int_0^t\d s\int_{-\infty}^\infty\d y\
			[p_{t-s}(y-x) - p_{t-s}(y-z)]^2\| Z(s\,,y)\|_k^2\\
		&\le 4c_0^2k\int_0^\infty\d s\int_{-\infty}^\infty\d y\
			[p_s(y-x+z) - p_s(y)]^2\\
		&=\frac{2c_0^2k}{\pi}\int_0^\infty\d s\int_{-\infty}^\infty\d\xi\
			\e^{-s\xi^2}\left| 1 - \e^{-i\xi(x-z)/2}\right|^2
			\qquad\text{[Plancherel's theorem]}\\
		&=\frac{8c_0^2k}{\pi}\int_0^\infty \frac{1-\cos (|x-z|\xi/2)}{\xi^2}\,\d \xi
			=2c_0^2k|x-z|.
	\end{align*}
	This proves the lemma.
\end{proof}

\begin{lemma}\label{lem:inc:t}
	Choose and fix  a non-random number $c_0>0$. Then,
	\[
		\adjustlimits
		\sup_{t,h>0}\sup_{x\in\R}\E\left( \left|
		\frac{I_Z(t+h\,,x) - I_Z(t\,,x)}{h^{1/4}}\right|^k \right) \le (5 c_0^2k)^{k/2},
	\]
	for every $k\in[2\,,\infty)$ and for all predictable random fields
	$Z$ that satisfy $\sup_{p\in\R_+\times\R}|Z(p)|\le c_0$.
\end{lemma}

\begin{proof}
	Choose and fix $t,h>0$ and $x\in\R$, and a predictable random field
	$Z$ as above, and then write
	\[
		\| I_Z(t+h\,,x) - I_Z(t\,,x)\|_k \le T_1 + T_2, 
	\]
	where
	\begin{align*}
		T_1 &= \left\| \int_{(0,t)\times\R} \left[ p_{t+h-s}(y-x) - p_{t-s}(y-x) \right]
			Z(s\,,y)\,W(\d s\,\d y) \right\|_k,\\
		T_2 &= \left\| \int_{(t,t+h)\times\R} p_{t+h-s}(y-x)Z(s\,,y)\,W(\d s\,\d y)\right\|_k.
	\end{align*}
	By the Burkholder-Davis-Gundy inequality in the form \cite{minicourse},
	for every real number $k\ge 2$,
	\begin{align*}
		T_1^2 &\le 4k\int_0^t\d s\int_{-\infty}^\infty\d y
			\left[ p_{t+h-s}(y-x) - p_{t-s}(y-x) \right]^2\|Z(s\,,y)\|_k^2\\
		&\le 4c_0^2k\int_0^\infty\d s\int_{-\infty}^\infty\d y
			\left[ p_{s+h}(y) - p_s(y) \right]^2\\
		&=\frac{2c_0^2k}{\pi}\int_0^\infty\d s\int_{-\infty}^\infty\d\xi\  \e^{-s\xi^2}
			\left| 1 - \e^{-h\xi^2/2}\right|^2 \qquad \qquad\qquad \text{[Plancherel's theorem]}\\
		&= \frac{2\sqrt{2}\,c_0^2k}{\pi}\int_0^\infty\frac{|1-\exp(-y^2)|^2}{y^2}\,\d y\,\sqrt h
			\le \frac{2\sqrt{2}\,c_0^2k}{\pi}\left( \frac13 + \int_1^\infty\frac{\d y}{y^2}\right)\sqrt h
			= \frac{8\sqrt{2}\,c_0^2k}{3\pi}\,\sqrt h,
	\end{align*}
	where we have used the bound $1-\exp(-y^2)\le y^2\wedge 1$ in order to obtain
	the last concrete numerical estimate. Similarly, we obtain
	\begin{align*}
		T_2^2 &\le 4k\int_t^{t+h}\d s\int_{-\infty}^\infty\d y\
			[p_{t+h-s}(y-x)]^2\|Z(s\,,y)\|_k^2\\
		&\le 4c_0^2k\int_0^h\d s\int_{-\infty}^\infty\d y\
			[p_{s+h}(y)]^2 = \frac{2 c_0^2k}{\pi}\int_h^{2h}\d s\int_{-\infty}^\infty\d\xi\
			\e^{-s\xi^2} \\
			&=\frac{2c_0^2k}{\sqrt\pi}\int_h^{2h}\frac{\d s}{\sqrt s}
			=\frac{4(\sqrt 2-1)c_0^2k}{\sqrt\pi}\sqrt h.
	\end{align*}
	We finally get
	\[
		\| I_Z(t+h\,,x) - I_Z(t\,,x)\|_k \le c_0\sqrt k \left[\sqrt{\frac{8\sqrt{2}}{3\pi}}
		+ \sqrt{\frac{4(\sqrt 2-1)}{\sqrt\pi}} \right]h^{1/4}\le
		2.1 c_0\sqrt{k}\, h^{1/4},
	\]
	and complete the proof, with room to spare for the constants of the inequality.
\end{proof}

Define 
\begin{equation*}
	\varrho(p) = |p_1|^{1/4} + |p_2|^{1/2}\qquad\text{for all $p=(p_1, p_2)\in\R^2$},
\end{equation*}
and for convenience, we use the following notation, $I_Z(p):=I_Z(p_1,p_2).$
\begin{lemma}\label{lem:inc:t,x}
	For every  non-random numbers $c_0,m>0$ and $\delta\in(0\,,1)$,
	\[
			\sup_{Z,\mathbb{I}}
			\E\exp\left( \alpha\sup_{\substack{p,q\in [0,1]\times\mathbb{I}\\
			0< \varrho(p-q)\le 1}}\left|
			\frac{I_Z(p) - I_Z(q)}{[\varrho(p-q)]^{1-\delta}}\right|^2 \right) <\infty,
	\]
	where $\sup_{Z,\mathbb{I}}$ denotes the supremum over all predictable random fields $Z$
	that satisfy $\sup_{p\in\R_+\times\R} |Z(p)|\le c_0$ and over all intervals $\mathbb{I}\subset\R$
	that have length $\le m$, and $\alpha$ is any positive number that satisfies
	\[
		\alpha < \frac{(1-2^{-\delta/2})^2}{2^{25}\e c_0^2}.
	\]
\end{lemma}

\begin{proof}
	Since $(a+b)^k\le 2^k(a^k+b^k)$ for all $k\ge1$ and $a,b\ge0$,
	Lemmas \ref{lem:inc:x} and \ref{lem:inc:t} together and Jensen's inequality imply that
	\begin{equation}\label{mom:est}\begin{split}
		\E\left( \left| \frac{I_Z(p)-I_Z(q)}{\varrho(p-q)}\right|^k\right)
		&\leq \left\{\E\left( \left| \frac{I_Z(p)-I_Z(q)}{\varrho(p-q)}\right|^{2k}\right)\right\}^{1/2} \\
		&\le   c_0^k 2^k(4^{k/2}+10^{k/2})k^{k/2}\leq (13 c_0)^k k^{k/2},
		\end{split}
	\end{equation}
	valid uniformly for all real numbers $k\ge1$, distinct $p,q\in\R_+\times\R$,
	and predictable $Z$ that satisfy $\sup_{p\in\R_+\times\R}|Z(p)|\le c_0$.
		
	We are going to use a suitable form of Garsia's lemma \cite[Appendix C]{Davar},
	and will begin by verifying the conditions that can be found in that reference.
	Note that  $\varrho(0)=0$ and $\varrho$ is subadditive:
	$\varrho(p+q)\le\varrho(p)+\varrho(q)$ for all $p,q\in\R^d$. 
	We use the notation of \cite[Appendix C]{Davar} and
	let 
	\[
		{\rm B}_\varrho(s) =\left\{ y\in\R^2:\, \varrho(y) \le s\right\}\qquad\text{for all $s\ge0$},
	\]
	and for all real numbers $k\ge1$,
	\[
		\I_k = \int_{[0,1]\times\mathbb{I}}\d p\int_{[0,1]\times\mathbb{I}}\d q\ 
		\left| \frac{I_Z(p)-I_Z(q)}{\varrho(p-q)}\right|^k.
	\]
	We know that $\I_k<\infty$ a.s.\ for every $k\ge1$. In fact, \eqref{mom:est} ensures that
	\begin{equation}\label{E(I_k)}
		\E(\I_k) \le m^2(13 c_0)^kk^{k/2},
	\end{equation}
	uniformly for all real numbers $k\ge1$, distinct $p,q\in\R_+\times\R$,
	and predictable $Z$ that satisfy $\sup_{p\in\R_+\times\R}|Z(p)|\le c_0$.
	If $(s\,,y)\in\R_+\times\R^2$ satisfies
	$|y_1|\le (s/2)^4$ and $|y_2|\le (s/2)^2$ then certainly $y\in B_\varrho(s)$. 
	Similarly, if $y\in B_\varrho(s)$, then certainly $|y_1|\le s^4$ and $|y_2|\le s^2$. 
	This argument shows
	that $(s/2)^6\le |B_\varrho(s)|\le 2s^6$ for all $s\ge0$,
	where $|\,\cdots|$ denotes the Lebesgue measure on $\R^2$.
	Consequently, $\int_0^{r_0}  |B_\varrho(s)|^{-2/k}\,\d s <\infty$
	for one, hence all, $r_0>0$, if and only if $k>12$ and
	\begin{align*}
		\int_0^{r_0}\frac{\d s}{|B_\varrho(s)|^{2/k}} &\le 2^{12/k}\int_0^{r_0}
			s^{-12/k}\,\d s \le \frac{2kr_0^{(k-12)/k}}{k-12}
			&\text{for every $r_0>0$ and $k>12$}\\
		&\le 4r_0^{(k-12)/k} 
			&\text{for every $r_0>0$ and $k\ge24$}.
	\end{align*}
	Apply Theorem C.4 of \cite{Davar} with $\mu(z)=z$ -- so that $C_\mu=2$ there --
	in order to see that
	\[
		\sup_{\substack{p,q\in[0,r]\times\mathbb{I}\\
		\varrho(p-q)\le r_0}}
		|I_Z(p) - I_Z(q)| \le 32\I_k^{1/k}\int_0^{r_0}\frac{\d s}{|B_\varrho(s)|^{2/k}}
		\le 128\I_k^{1/k}r_0^{(k-12)/k}
		\qquad\text{a.s.},
	\]
	for every nonrandom $k\ge24$ and $r_0>0$. In particular, we learn from \eqref{E(I_k)}
	that
	\[
		\E\left( \sup_{\substack{p,q\in[0,1]\times\mathbb{I}\\
		\varrho(p-q)\le r_0}}
		|I_Z(p) - I_Z(q)|^k \right) \le 128^k r_0^{k-12}\E(\I_k)
		\le m^2(1664 c_0)^k   r_0^{k-12}k^{k/2},
	\]
	for every  $k\ge24$ and $r_0>0$, as well as all $r>0$, all intervals $\mathbb{I}$ of
	length $m$, and all predictable fields $Z$ that satisfy $\sup_{p\in\R_+\times\R}|Z(p)|\le c_0$.
	We freeze all variables and define for every $\delta\in(0\,,1)$ and $n\in\mathbb{Z}_+$,
	\[
		S_{n,\delta} = \left\{\E\left( \sup_{\substack{p,q\in[0,1]\times\mathbb{I}\\
		2^{-n-1}< \varrho(p-q)\le 2^{-n}}}
		\left| \frac{I_Z(p) - I_Z(q)}{[\varrho(p-q)]^{1-\delta}} \right|^k \right) \right\}^{1/k}.
	\]
	It follows that as long as $k\ge 24$,
	\[
		S_{n,\delta} \le 2^{(1-\delta)(n+1)} 
		\left\{ \E\left( \sup_{\substack{p,q\in[0,1]\times\mathbb{I}\\
		\varrho(p-q)\le 2^{-n}}}
		|I_Z(p) - I_Z(q)|^k \right) \right\}^{1/k}
		\le 2^{12-\delta} c_0 m^{2/k} 2^{-n[\delta-(12/k)]} \sqrt k.
	\]
	Sum the preceding over all $n\in\mathbb{Z}_+$ to see that, as long as $k\ge (24/\delta)>(12/\delta)\vee 24$,
	\[
		\left\{\E\left( \sup_{\substack{p,q\in[0,1]\times\mathbb{I}\\
		\varrho(p-q)\le 1}}
		\left| \frac{I_Z(p) - I_Z(q)}{[\varrho(p-q)]^{1-\delta}} \right|^k \right) \right\}^{1/k}
		\le \frac{2^{12-\delta} c_0 m^{2/k} \sqrt k}{1 - 2^{-[\delta-(12/k)]}}
		\le \frac{2^{12}}{1-2^{-\delta/2}}c_0 m^{2/k} \sqrt k.
	\]
	Replace $k$ by $2k$ and restrict attention to integral choices of $k$ in order to see that
	\[
		\E\left( \sup_{\substack{p,q\in[0,1]\times\mathbb{I}\\
		\varrho(p-q)\le 1}}
		\left| \frac{I_Z(p) - I_Z(q)}{[\varrho(p-q)]^{1-\delta}} \right|^{2k} \right)
		\le m^2 \left(\frac{2^{25/2}\sqrt\e\, c_0}{1-2^{-\delta/2}}\right)^{2k}k!
		=: m^2 Q^k k!,
	\]
	for every integer  $k \ge 12/\delta$, as well as all $r>0$, all intervals $\mathbb{I}$ of
	length $m$, and all predictable fields $Z$ that satisfy $\sup_{p\in\R_+\times\R}|Z(p)|\le c_0$,
	where where we have used the inequality
	$k^k\le \e^k k!$ valid for all positive integers $k$. An appeal to the Taylor series
	expansion of the exponential function $v\mapsto \exp(\alpha v^2)$ yields 
	\[
		\E\exp\left( \alpha\sup_{\substack{p,q\in[0,1]\times\mathbb{I}\\
		\varrho(p-q)\le 1}}
		\left| \frac{I_Z(p) - I_Z(q)}{[\varrho(p-q)]^{1-\delta}} \right|^2\right)\le
		\frac{m^2}{1-\alpha Q}<\infty,
	\]
	for every $\alpha$ that satisfies $\alpha<Q^{-1}$. This proves the lemma.
\end{proof}

We are ready to conclude this section.

\begin{proof}[Proof of Theorem \ref{th:key}]
	Lemma \ref{lem:tails} ensures that
	\[
		\P\left\{ I_Z(a\,,c) > M\left(\frac{a}{\pi}\right)^{1/4}\right\} \ge \frac{C_1\e^{-M^2/(2c_1^2)}}{1+M},
	\]
	uniformly for all $a>0$, $c\in\R$, and $M\ge1$. In particular,
	\begin{align*}
		&\P\left\{ \inf_{t\in(a,a+\varepsilon^4)}\inf_{x\in(c,c+\varepsilon^2)} I_Z(t\,,x)
			\le  M\left(\frac{a}{\pi}\right)^{1/4}\right\}\\
		&\qquad\le 1 - \frac{C_1\e^{-(2M)^2/(2c_1^2)}}{1+2M}
			+ \P\left\{ \adjustlimits\sup_{t\in(a,a+\varepsilon^4)}\sup_{x\in(c,c+\varepsilon^2)}
			| I_Z(t\,,x) -  I_Z(a\,,c)| \ge M\left(\frac{a}{\pi}\right)^{1/4}\right\}.
	\end{align*}
	Chebyshev's inequality yields the following:
	\begin{align*}
		&\P\left\{ \adjustlimits\sup_{t\in(a,a+\varepsilon^4)}\sup_{x\in(c,c+\varepsilon^2)}
			| I_Z(t\,,x) - I_Z(a\,,c)| \ge M\left(\frac{a}{\pi}\right)^{1/4}\right\}\\
		&\le\P\left\{ \adjustlimits\sup_{t\in(a,a+\varepsilon^4)}\sup_{x\in(c,c+\varepsilon^2)}
			\left| \frac{I_Z(t\,,x) - I_Z(a\,,c)}{\sqrt{\varrho\left( (t\,,x) - (a\,,c)\right)}} \right| 
			\ge \frac{M (a/\pi)^{1/4}}{\sqrt{2\varepsilon}}\right\}\\
		&\le \E\exp\left(\alpha \sup_{t\in(a,a+\varepsilon^4)}\sup_{x\in(c,c+\varepsilon^2)}
			\left|\frac{ I_Z(t\,,x) - I_Z(a\,,c)}{\sqrt{\varrho((t\,,x)-(a\,,c))}}\right|^2\right)
			\times\exp\left( - \frac{\alpha M^2 \sqrt{a/\pi}}{2\varepsilon}\right),
	\end{align*}
	uniformly for all $M\ge1$ and $a,c,\varepsilon,\alpha>0$. Choose and fix
	\begin{equation}\label{alpha:epsilon}
		\alpha = \frac{(1-2^{-1/4})^2}{2^{26}\e (c_1\vee c_2)^2}
		\quad\text{and}\quad
		\varepsilon =\frac{c_1^2\alpha}{8}\sqrt{\frac{a}{\pi}}.
	\end{equation}
	and apply Lemma \ref{lem:inc:t,x} [with $\delta=\frac12$ and
	$c_0=c_1\vee c_2$] 
	in order to see that there exists $K = K(c_1\,,c_2)>1$ such that
	\begin{align*}
		\P\left\{ \inf_{t\in(a,a+\varepsilon^4)}\inf_{x\in(c,c+\varepsilon^2)} I_Z(t\,,x)
			\le  M\left(\frac{a}{\pi}\right)^{1/4}\right\}
			&\le 1 - \frac{C_1\e^{-(2M)^2/(2c_1^2)}}{1+2M} + K\e^{-(2M)^2/c_1^2}\\
		&\le 1- \e^{-(2M)^2/(2c_1^2)}\left[\frac{C_1}{3M}
			-K\e^{-(2M)^2/(2c_1^2)}\right],
	\end{align*}
	uniformly for all $M\ge1$ and $a>0$. In particular, there exists $M_0=M_0(c_1\,,c_2)>1$
	such that for all $M\ge1$ and $a>0$,
	\[
		\sup_{a,c>0}
		\P\left\{ \inf_{t\in(a,a+\varepsilon^4)}\inf_{x\in(c,c+\varepsilon^2)} I_Z(t\,,x)
		\le  M\left(\frac{a}{\pi}\right)^{1/4}\right\}
		\le 1 - \frac{C_1\e^{-(2M)^2/(2c_1^2)}}{6M}
	\]
	uniformly for all $M\ge M_0$. To be sure, we remind also that $\varepsilon=\varepsilon(a\,,c_1\,,c_2)$
	is defined in \eqref{alpha:epsilon}.
	In any case, this readily yields
	\begin{equation}\label{P>}
		\inf_{a>0}\P\left\{ \limsup_{c\to\infty}
		\inf_{t\in(a,a+\varepsilon^4)}\inf_{x\in(c,c+\varepsilon^2)}I_Z(t\,,x)
		>  M\left(\frac{a}{\pi}\right)^{1/4}\right\}
		\ge \frac{C_1\e^{-(2M)^2/(2c_1^2)}}{6M}>0,
	\end{equation}
	uniformly for all $M\ge M_0$. 
	Since we are assuming that the infinite-dimensional
	process $x\mapsto I_Z(\cdot\,,x)$ is ergodic, 
	we can improve \eqref{P>} to the following without need for additional
	work:	
	\[
		\P\left\{ \limsup_{c\to\infty}
		\inf_{t\in(a,a+\varepsilon^4)}\inf_{x\in(c,c+\varepsilon^2)} I_Z(t\,,x)
		>  M\left(\frac{a}{\pi}\right)^{1/4}\right\}=1,
	\]
	uniformly for all $M\ge M_0$ and $a>0$. 
	We now can send $M\to\infty$ to deduce the theorem from the particular
	form of $\varepsilon$ that is given in \eqref{alpha:epsilon}.
\end{proof}

\section{Ergodicity of the solution}

In this section, we consider equation \eqref{SHE} with constant initial condition $\rho \in \R$.
That is,
\begin{equation}\label{mild2}
	u(t\,,x) = \rho + \int_{(0,t)\times\R} p_{t-s}(y-x) b(u(s\,,y))\,\d s\,\d y
	+ \I(t\,,x),
\end{equation}
where 
\begin{equation*}
	\I(t\,,x) = \int_{(0,t)\times\R} p_{t-s}(y-x) \sigma(u(s\,,y)) \, W(\d s\,\d y).
\end{equation*}

The aim of this section is to show that when $\sigma$ and 
$b$ are  Lipschitz continuous the solution to (\ref{mild2}) is spatially ergodic.
This follows from an application of Theorem \ref{th:key}.  Note that because
we are discussing Lipschitz continuous $b$, there is no need to describe 
what we mean by solution; that is done already in Walsh \cite{Walsh}.

According to Bally and Pardoux \cite{BP} (see also Nualart \cite[Proposition 1.2.4]{N}), under these conditions
$u(t\,,x) \in \mathbb{D}^{1,P}$
for all $p \geq 2$, $t>0$, and $x \in \R$, and the Malliavin derivative $Du(t\,,x)$ satisfies
\begin{equation*} \begin{split}
	D_{r,z} u(t\,,x) = p_{t-r}(x-z) \sigma(u(r,z)) &+ 
		\int_{(r,t)\times\R} p_{t-s}(y-x) B_{s,y} D_{r,z}u(s\,,y) \,\d s\,\d y\\
	&+ \int_{(r,t)\times\R} p_{t-s}(y-x) \Sigma_{s,y} D_{r,z} u(s,y)  \, W(\d s\,\d y) \qquad \text{a.s},
\end{split}\end{equation*}
for a.e.\ $(r\,,z)\in(0\,,t)\times\R$ where $B$ and $\Sigma$ are a.s. bounded random fields.  We have the following estimate on the Malliavin derivative.

\begin{lemma} \label{p}
	If $\sigma$ and $b$ are Lipschitz continuous, then for every $T>0$ and $p \geq 2$ 
	there exists $C_{T,p}>0$ such that
	\[
		\Vert D_{r,z} u(t\,,x) \Vert_p \leq C_{T,p}\ p_{t-r}(x-z) p_r(z).
	\]
	uniformly for $t \in (0\,,T)$ and $x \in \R$, and for almost every $(r\,,z) \in (0\,,t) \times \R$.
\end{lemma}

\begin{proof}
	The proof follows closely the proof of Lemma 2.1 in Chen et al \cite{CKNP}
	but we must account for a few of the changes that are caused by the drift $b$:
	By Minkowski's inequality,
	\begin{equation*} \begin{split}
		\bigg\Vert \int_{(r,t)\times\R} p_{t-s}(y-x) B_{s,y} D_{r,z}u(s\,,y) \,\d s\,\d y\bigg\Vert_p^2
		\leq c \int_r^t \d s\int_{-\infty}^\infty\d y \left[p_{t-s}(x-y)\right]^2 
		\Vert D_{r,z} u(s\,,y) \Vert^2_p.
	\end{split}\end{equation*}
	This is the same expression that appears in the right-hand side of (2.6) in \cite{CKNP}.
	Therefore, the rest of the proof follows the analogous argument in \cite[Proof of Lemma 2.1]{CKNP}.
\end{proof}

We are now ready to state the main result of this section.
\begin{corollary} \label{th_erg}
	If $\sigma$ and $b$ are Lipschitz continuous,  then the random fields $x \rightarrow u(t\,,x)$ 
	and $x \rightarrow \mathcal{I}(t\,,x)$ are  stationary and ergodic for every $t>0$.
\end{corollary}

\begin{proof}
	Stationarity follows from Chen et al \cite[Lemma 7.1]{CKNP1}, and 
	ergodicity is a direct consequence of Lemma \ref{p} and Theorem \ref{eq:I_Z}.
\end{proof}

\section{A lower bound via differential inequalities}
In this section, we continue to assume that $b$ is  Lipschitz continuous. 
Our aim is to prove the following key result.

\begin{theorem}\label{th:Lip}
	If $b:\R\to(0\,,\infty)$ is  Lipschitz continuous,
	then for every non-random number $a>0$,
	there exists a non-random number $\varepsilon = \varepsilon(a)>0$ -- not
	depending on the choice of $b$ --  that
	satisfies the following for every $M>\|u_0\|_{L^\infty(\R)}$: 
	There exists an a.s.-finite random variable $c = c(a\,,M)>0$ such that
	\[
		\inf_{t\in[a+\varepsilon,a+2\varepsilon]}\inf_{x\in(c,c+\sqrt\varepsilon)} u(t\,,x) \ge 
		\sup\left\{N>M:\, \int_{M+\rho}^{N+\rho}
		\frac{\d y}{b(y)}<\varepsilon\right\}
		\qquad\text{a.s.}
		\quad[\sup\varnothing=0],
	\]
	where $\rho: = \inf_{x\in\R} u_0(x)$.
\end{theorem}

The following result will be useful for the proof of the above theorem.

\begin{lemma}\label{lem:LV}
	Fix two numbers $N>A>0$ and suppose $B:\R_+\to(0\,,\infty)$ is Lipschitz continuous.
	Let $T=\int_A^N\d s/B(s)$, and suppose that $F:\R_+\to\R_+$ solves
	\[
		F(t) \ge A+\int_0^t B(F(s))\,\d s\qquad\text{for all $t\in[0\,,2T]$}.
	\]
	Then $\inf_{t\in[T,2T]}F(t)\ge N$.
\end{lemma}

\begin{remark}\label{rem:diff_ineq}
	Lemma \ref{lem:LV} can recast in slightly weaker terms as
	a statement about the differential inequality,
	\begin{equation*}\left[\begin{split}
		&F'\ge B\circ F\qquad\text{ on $\R_+$,}\\
		&\text{subject to $F(0)\ge A$}.
	\end{split}\right.\end{equation*}
	In this case, $F(t)\ge N$ some time $t$ between $T=\int_A^N\d s/B(s)$ and time $2T$.
\end{remark}

\begin{proof}
	Choose and fix an $A>0$. 
	The ordinary differential equation $G(t)=A+\int_0^t B(G(s))\,\d s$ has a unique 
	continuous solution 
	that is strictly increasing (hence also  has an inverse) up to
	time $T = \sup\{ t>0:\, G(t)\le N\}$ for every $N>A$, and
	$G(T)=\lim_{s\uparrow T}G(s)=N$. We also have that $G(t) \geq N$ for all $t \in [T, 2T]$. A comparison theorem yields $F\ge G$ on $[0\,,2T]$,
	and completes the proof.
\end{proof}

\begin{proof}[Proof of Theorem \ref{th:Lip}]
We first assume that the initial data is equal to a constant $\rho \in \R$.	Choose and fix $a>0$.
	According to Corollary \ref{th_erg} and Theorem \ref{th:key}, we can associate to $a$ a non-random
	number $\varepsilon>0$ such that
	$$\limsup_{c\to\infty}\inf_{t\in(a+\varepsilon,a+2\varepsilon)}\inf_{x\in(0,\sqrt\varepsilon)}\I(t\,,c+x)=\infty, \quad \text{a.s.}$$
	
	Also choose and fix a number $M>0$. According to Theorem \ref{th:key},
	we can find a random number $c>0$ such that
	\begin{equation}\label{I>}
		\adjustlimits\inf_{t\in(a+\varepsilon,a+2\varepsilon)}\inf_{x\in(0,\sqrt\varepsilon)}\I(t\,,c+x) > M
		\quad\text{a.s.}
	\end{equation}
	Because $b\ge0$ and $b$ is nondecreasing,
	\begin{align*}
		u(a+t\,,c+x) &\ge \rho + 
			\int_{(0,t)\times\R} p_{a+t-s}(y-x-c) b(u(s\,,y))\,\d s\,\d y
			+ \I(a+t\,,c+x)\\
		&\ge \rho + \int_{(0,t)\times\R} p_{t-s}(y-x) b(u(a+s\,,c+y))\,\d s\,\d y
			+ \I(a+t\,,c+x)\\
		&\ge \rho+ 
			\int_0^tb\left(\inf_{z\in(0,\sqrt\varepsilon)}u(a+s\,,c+z)\right)\d s\int_0^{\sqrt\varepsilon}\d y\ 
			p_{t-s}(y-x)  + \I(a+t\,,c+x),
	\end{align*}
	a.s., for every $t,c>0$ and $x\in\R$. If in addition $x\in(0\,,\sqrt\varepsilon)$
	and $t\in(0\,,2\varepsilon)$, then
	\[
		\int_0^{\sqrt\varepsilon}p_{t-s}(y-x)\,\d y
		= \int_{-x}^{-x+\sqrt\varepsilon} p_{t-s}(y)\,\d y
		\ge\int_{-\sqrt{\varepsilon}}^0 p_{t-s}(y)\,\d y
		\ge\int_{-1/2}^0 p_1(y)\,\d y=\ell>0,
	\]
	uniformly for all $s\in(0\,,t)$. Therefore, \eqref{I>} tells us that,
	for all $x\in(0\,,\sqrt\varepsilon)$
	and $t\in(0\,,2\varepsilon)$,
	\[
		u(a+t\,,c+x) \ge \ell \int_0^tb\left(\inf_{z\in(0,\sqrt\varepsilon)}u(a+s\,,c+z)\right)\d s
		+ M+\rho.
	\]
	In other words, we have shown that the function
	\[
		f(t) = \inf_{x\in(0\,,\sqrt\varepsilon)} u(a+t\,,c+x)
		\qquad[t>0]
	\]
	satisfies
	\[
		f(t) \ge M +\rho + \ell  \int_0^t b(f(s))\,\d s
		\quad\text{uniformly for all $t\in(0\,,2\varepsilon)$}.
	\]
	Since $\int_{M+\rho}^{N+\rho} [b(y)]^{-1}\,\d y<\varepsilon$,
	Lemma \ref{lem:LV} assures us that $\inf_{t\in[\varepsilon,2\varepsilon]}f(t) \ge N$. Hence,
	\[
		\adjustlimits\inf_{s\in[a+\epsilon,a+2\epsilon]}\inf_{y\in(c\,,c+\sqrt\varepsilon)} u(s\,,y) \ge N
		\quad\text{a.s.},
	\]
	which yields the theorem in the case that the initial data is constant. For the general case that the initial condition is bounded, using a standard comparison theorem we can deduce the proof of the theorem.
\end{proof}

\section{Minimal solutions, and proof of Theorem \ref{th:blowup}}

We begin by revisiting the well posedness of \eqref{SHE} under
Assumptions \ref{cond-dif} and \ref{cond-drift}.  After that, we
prove Theorem \ref{th:blowup} and conclude the paper.

\subsection{Minimal solutions}

Let $\Lloc$ denote the collection of all functions
$f:\R\to(0\,,\infty)$ that are nondecreasing and locally Lipschitz continuous.
In particular, Assumption \ref{cond-drift} is shortened to the assertion that $b\in\Lloc$.
We also define $\L$ to be the collection of all elements of $\Lloc$ that are
Lipschitz continuous. 

Throughout this subsection, we write the solution to \eqref{SHE} as $u_b$ provided
that \eqref{SHE} well posed for a given $b\in\Lloc.$ As a consequence of the theory of
Walsh \cite{Walsh}, \eqref{SHE} is well posed for example when $b\in\L$;
see also Dalang \cite{Dalang}. Moreover,
$u_b$ is defined uniquely provided additionally that 
$\sup_{t\in(0,T)}\sup_{x\in\R}\|u(t\,,x)\|_2<\infty$ for all $T>0$. Finally,
\[
	\P\{ u_b\le u_c\}=1\qquad\text{for all $b,c\in\L$ that satisfy $b\le c$};
\]
see Mueller \cite{Mueller} and \cite{GeisMan}.

Now suppose that $b\in\Lloc$, as is the case in the Introduction.
Let $b^{(n)}=b\wedge n$ for every $n\in\mathbb{N}$. The monotonicity of $b$
implies that every $b^{(n)}\in \L$ for every $n\in\mathbb{N}$,
and $b^{(n)} \leq b^{(m)}$ when $n \leq m$. Since $u_{b^{(n)}} \le u_{b^{(m)}}$
whenever $n\le m$, it follows that the random field
\[
	u = \lim_{n\to\infty} u_{b^{(n)}}
\]
exists and has lower-semicontinuous sample functions. Note also that if
$c\in\L$ satisfies $c\le b$, then $u_c\le u$. This proves immediately that
\[
	u = \sup_{c\in\L} u_c.
\]
Therefore, we refer to $u$ as the \emph{minimal solution} to \eqref{SHE} when $b$
satisfies Assumption \ref{cond-drift}. 

Next we describe why $u$ can justifiably be called the minimal ``solution''
to \eqref{SHE}. Minimality is clear from context. However, ``solution''
deserves some words.

If $b$ is in addition Lipschitz continuous, then
$u$ is the solution to \eqref{SHE} that the Walsh theory yields and there is nothing to discuss.
Now suppose $b\in\Lloc$ and recall $b^{(n)}\in\L$. We may observe that
\[
	b^{(n)}\left( u_{b^{(n)}}(t\,,x)\right) \le 
	b^{(m)}\left( u_{b^{(m)}}(t\,,x)\right)\qquad\text{whenever $n\le m$},
\]
off a single null set that does not depend on $(b\,,n\,,m)$. Since
\[
	b^{(n)}(x)=\frac{b(x)+n-|b(x)-n|}{2},
\]
it follows that
\begin{equation}\label{bnb}
	\lim_{n\to\infty} b^{(n)}\left( u_{b^{(n)}}(t\,,x)\right) = b(u(t\,,x))
	\qquad\text{for all $t>0$ and $x\in\R$},
\end{equation}
again off a single null set [these are real-variable, sure, assertions]. Therefore,
the monotone convergence theorem yields
\[
	\lim_{n \rightarrow \infty} \int_{(0,t)\times\R} p_{t-s}(y-x) b^{(n)}(u^{(n)}(s\,,y))\,\d s\,\d y=
	\int_{(0,t)\times\R} p_{t-s}(y-x) b(u(s\,,y))\,\d s\,\d y,
\]
where $b(\infty)=\sup b$.

Next, let us consider the $[0\,,\infty]$-valued random variable 
\[
	\tau = \inf\left\{ t>0:\,
	u(s\,,y)=\infty \quad\text{for all $s\le t$ and $y\in\R$}\right\},
\]
where $\inf\varnothing=0$. Because $u$ is lower semicontinuous, 
one can show that $\tau$ is a stopping time with respect to the filtration of
the noise, which we assume satisfies the usual conditions of martingale
theory, without loss of generality. Of course, $\tau$ is the first blowup time of
$u$. Since $\sigma$ is a bounded and continuous function,
\begin{align*}
	&\lim_{n\rightarrow \infty}\left\|\int_{(0,t\wedge\tau)\times\R} 
		p_{t-s}(y-x) [\sigma(u^{(n)}(s\,,y))-\sigma(u(s\,,y))]\,W(\d s\,\d y)\right\|_2^2\\
	&=\E \left(\int_{(0,t\wedge\tau)\times\R} \left[p_{(t\wedge\tau)-s}(y-x)\right]^2 
		\lim_{n\rightarrow \infty}[\sigma(u^{(n)}(s\,,y))-\sigma(u(s\,,y))]^2 \d s\, \d y\right)=0,
\end{align*}
where $\int_\varnothing(\,\cdots)=0$. Taken together, these comments prove that
if $\tau>0$ -- that is if the solution to \eqref{SHE} does not instantly blow up -- then
$u$ satisfies \eqref{mild} for all $x\in\R$ and all times $t<\tau$.\footnote{In fact, one can show that
the $\liminf$ of the stochastic integrals in the mild formulation of $u^{(n)}$ is finite a.s. See
the end of the proof of Theorem \ref{th:blowup}. This implies the stronger statement that,
for all $t>0$ and $x\in\R$,
\[
	u(t\,,x) = (p_t*u_0)(x) + \int_{(0,t)\times\R} p_{t-s}(y-x) b(u(s\,,y))\,\d s\,\d y + \text{a finite term},
\]
where $b(\infty)=\sup b$. Theorem \ref{th:blowup} implies that both sides of the above identity
are infinite when \eqref{Osgood} holds.} In this sense, our
extension of the solution theory of Walsh \cite{Walsh} indeed produces solutions for $b\in\Lloc$ if
there is chance for non-instantaneous blowup, and the smallest such solution is $u$.

Theorem \ref{th:blowup} says that if $b\in\Lloc$ satisfies the Osgood condition
\eqref{Osgood}, then the minimal solution
satisfies $u(t)\equiv\infty$ for all $t>0$. 

Now suppose the Osgood condition holds,
and consider any solution theory that extends the Walsh theory and has a comparison theorem.
The preceding comments prove that if that solution theory produces a solution $v$, then that
solution satisfies $u\le v$ and hence
$v(t)\equiv\infty$ for all $t>0$ by Theorem \ref{th:blowup}. This is the precise conditional sense
in which Theorem \ref{th:blowup} says that ``the solution'' to \eqref{SHE}
blows up instantaneously and everywhere.

We can now conclude the paper with the following.

\subsection{Proof of Theorem \ref{th:blowup}}
We now return to the notation of the Introduction and write $u$ in place of $u_b$,
and prove the everywhere and instantaneous blow up of $u$ under \eqref{Osgood},
where the symbol $u$ denotes
the minimal solution to \eqref{SHE} as was described in the previous subsection. 

Choose and fix $a>0$,
and in light of \eqref{Osgood} we may choose and fix $M>\|u_0\|_{L^\infty(\R)}$ such that
\[
	\int_{M+\rho}^\infty\frac{\d y}{b(y)}<\varepsilon. 
\]
Then, the construction
of $u$ and Theorem \ref{th:Lip} together yield $\varepsilon$ -- independently of the choice
of $b$ and $M$ -- such that the following
holds for every $n\in\mathbb{N}$:
\begin{align*}
	\adjustlimits\inf_{t\in[a+\varepsilon,a+2\varepsilon]}\inf_{x\in(c,c+\sqrt\varepsilon)} u(t\,,x) &\ge
		\adjustlimits\inf_{t\in(a+\varepsilon,a+2\varepsilon)}\inf_{x\in(c,c+\sqrt\varepsilon)} 
		u^{(n)}(t\,,x)\\
	&\ge \sup\left\{N>M:\, \int_{M+\rho}^{N+\rho} 
		\frac{\d y}{b^{(n)}(y)}<\varepsilon\right\}
		\qquad\text{a.s.}
\end{align*}
Let $n\uparrow\infty$ to see from the monotone convergence theorem that
\[
	\adjustlimits\inf_{t\in[a+\varepsilon,a+2\varepsilon]}\inf_{x\in(c,c+\sqrt\varepsilon)} u(t\,,x) \ge
	\sup\left\{N>M:\, \int_{M+\rho}^{N+\rho } 
	\frac{\d y}{b(y)}<\varepsilon\right\}=\infty
	\qquad\text{a.s.}
\]
This proves that the blowup time is a.s.\ $\le a+2\varepsilon(a)$ 
and that the solution blows up everywhere
in a random interval of the type $(c\,,c+\sqrt\varepsilon)$. Consequently,
for every non-random $t>0$ there a.s.\ is a random closed interval $I(t) \subset (0\,, \infty)$
and and a non-random closed interval $\tilde{I}(t)=[a+\varepsilon\,, a+2\varepsilon]\subset (0\,,t)$ such that
\begin{equation}\label{inf:blowup}
	\inf_{(s,x)\in \tilde{I}(t) \times I(t)} u(s\,,x)=\infty \qquad \text{a.s.} 
\end{equation}
We now consider the process  $u^{(n)} = u_{b^{(n)}}$, as defined in the previous subsection.
For every $n\in\mathbb{N}$, the random field $u^{(n)}$ solves
	\begin{equation*} \begin{split}
	u^{(n)}(t\,,x) = (p_t*u_0)(x) &+ \int_{(0,t)\times\R} p_{t-s}(y-x) b^{(n)}(u^{(n)}(s\,,y))\,\d s\,\d y\\
	&+ \int_{(0,t)\times\R} p_{t-s}(y-x) \sigma(u^{(n)}(s,y)) \, W(\d s\,\d y).
	\end{split}
\end{equation*}
By the monotone convergence theorem,
\[
	\int_{(0,t)\times\mathbb{R}} p_{t-s}(y-x) b^{(n)}(u^{(n)}(s\,,y))\, \d s\,\d y 
	\geq \int_{\tilde{I}(t)\times I(t)} p_{t-s}(y-x) b^{(n)}(u^{(n)}(s\,,y))\, \d s\,\d y
	\uparrow\infty,
\]
as $n\to\infty$; see \eqref{bnb} and \eqref{inf:blowup}. At the same time, standard estimates 
such as those in \S2 show that 
\[
	\sup_{n\in\mathbb{N}} 
	\E\left( \sup_{(t,x)\in K}\left| \int_{(0,t)\times\mathbb{R}} 
	p_{t-s}(y-x)\sigma(u^{(n)}(s,y))\, W(\d s\,\d y)\right|^2\right)<\infty,
\]
for every compact set $K\subset\mathbb{R}_+\times\mathbb{R}$. Therefore, 
Fatou's lemma ensures that a.s.,
\[
	\liminf_{n\to\infty}\sup_{(t,x)\in K} \int_{(0,t)\times\mathbb{R}} p_{t-s}(y-x)
	\sigma(u^{(n)}(s,y))\, W(\d s\,\d y)<\infty.
\]
It follows that $\inf_K u=\infty$ a.s.\
for all compact sets $K\subset\mathbb{R}_+\times\mathbb{R}$.
This concludes the proof.
\qed\\

\noindent\textbf{Acknowledgement.} We are grateful to  Alison
Etheridge for her questions that ultimately led to this paper, 
and for sharing her insights with us. 

\bibliography{Foon-Nual}

\end{document}